\documentclass[10pt, a4paper]{amsart}

\usepackage[T1]{fontenc}
\usepackage[utf8]{inputenc}
\usepackage{amsmath, amssymb, amsfonts, amsthm, amsbsy}
\usepackage[margin=2.0cm]{geometry}
\usepackage[]{setspace}
\usepackage[shortlabels]{enumitem}
\usepackage[usenames,dvipsnames]{xcolor}
\usepackage[,unicode]{hyperref}
\hypersetup{
	urlcolor=blue,
	colorlinks=true,
	linkcolor= blue,
	citecolor=blue
}
\usepackage[all]{xy}
\usepackage{libertine,libertinust1math}
\usepackage{tikz}
\usetikzlibrary{fit,shapes,calc,arrows,through,intersections,arrows.meta,decorations.pathmorphing}
\usepackage{tkz-euclide}
\theoremstyle{plain}
\newtheorem{theorem}{Theorem}[section]
\newtheorem*{theorem*}{Theorem}
\newtheorem*{corollary*}{Corollary}
\newtheorem{introtheorem}{Theorem}

\newtheorem{proposition}[theorem]{Proposition}
\newtheorem{lemma}[theorem]{Lemma}
\newtheorem{corollary}[theorem]{Corollary}

\newtheorem{conjecture}[theorem]{Conjecture}
\theoremstyle{definition}
\newtheorem{example}[theorem]{Example}
\newtheorem{definition}[theorem]{Definition}
\newtheorem{observation}[theorem]{Observation}
\newtheorem{question}[theorem]{Question}

\theoremstyle{remark}
\newtheorem{remark}[theorem]{Remark}

\newcommand\PP{\mathrm{P}}

\newcommand\ac{\mathrm{ac}}
\newcommand\res{\mathrm{res}}
\newcommand\ch{\mathrm{char}}
\newcommand\Lang{\mathfrak{L}}
\newcommand\Lring{\Lang_{\mathrm{ring}}}

\newcommand\Lfield{\Lang_{\mathrm{field}}}
\newcommand\Lval{\mathfrak{L}_{\mathrm{val}}}
\newcommand\Lvalpi{\Lval(\varpi)}
\newcommand\Loag{\Lang_{\mathrm{oag}}}

\newcommand\Lac{\Lang_{\ac}}

\newcommand\Lk{\Lang_{\bbk}}
\newcommand\LG{\Lang_{\bbG}}

\newcommand\Llambda{\Lang_{\lambda}}
\newcommand\Llambdaring{\Lang_{\lambda,\mathrm{ring}}}
\newcommand\Llambdaval{\Lang_{\lambda,\mathrm{val}}}
\newcommand\ik{\iota_{\bbk}}
\newcommand\iG{\iota_{\bbG}}
\newcommand\psvar[2]{#1(\!({#2})\!)}
\newcommand\Hs[2]{\psvar{#1}{t^{#2}}}
\newcommand\ps[1]{\Hs{#1}{}}
\newcommand\ipsvar[2]{#1[\![{#2}]\!]}
\newcommand\iHs[2]{\ipsvar{#1}{t^{#2}}}
\newcommand\ips[1]{\iHs{#1}{}}
\renewcommand\mid{:}
\newcommand\TVF{\mathsf{TVF}}
\newcommand\STVF{\mathsf{S}\TVF}

\renewcommand\Form{\mathrm{Form}}

\renewcommand\H{\mathrm{H}}
\newcommand\Hen{\H^{\mathrm{e}\prime}} % Henselian equicharacteristic nontrivial - class
\newcommand\Th{\mathrm{Th}}
\newcommand\tp{\mathrm{tp}}

\newcommand\reg{\mathrm{reg}}
\newcommand\impdeg{\mathfrak{imp}}
\newcommand\Impdeg{\mathfrak{Imp}}

\newcommand\frakI{\mathfrak{I}}
\newcommand\bbval{\mathbb{v}}
\newcommand\bbres{\mathbb{res}}
\newcommand\bbK{\mathbb{K}}
\newcommand\bbk{\mathbb{k}}
\newcommand\bbG{\mathbb{G}}
\newcommand{\Gal}{\mathrm{Gal}}
\newcommand{\lGal}{\lambda,\mathrm{Gal}}

\newcommand{\Pos}{\mathrm{Pos}}

\newcommand{\Frag}{\mathrm{F}}
\newcommand{\FEone}{\exists_{1}} %%% E_1 FRAGMENT
\newcommand{\Epos}{[\exists_{1}^{\Pos}]} %%% E_1 positive not quite FRAGMENT
\newcommand{\FEpos}{\exists_{1}^{\Pos}} %%% E_1 positive FRAGMENT
\newcommand{\FGal}{\Frag_{\Gal}} %%% GALOIS FRAGMENT
\newcommand{\FlGal}{\Frag_{\lGal}} %%% LAMBDA GALOIS FRAGMENT
\newcommand\FG{\Frag_{\bbG}}
\newcommand\Fk{\Frag_{\bbk}}
\newcommand{\Tred}{\leq_{\rm T}}

\newcommand{\Teq}{\simeq_{\rm T}}
\newcommand{\Toplus}{\oplus_{T}}
\newcommand{\Deq}{\Teq}
\newcommand{\Doplus}{\Toplus}
\newcommand{\Dtext}{Turing}
\newcommand\FF{\mathbb{F}}
\newcommand\NN{\mathbb{N}}
\newcommand\ZZ{\mathbb{Z}}
\newcommand\QQ{\mathbb{Q}}

\newcommand\SAcut[1]{%
	}

\newcommand\Alg{$(\blacktriangle)$}

\newcommand\define[1]{{\bf#1}}
\newcommand\resp[1]{{\rm[}#1{\rm ]}}

\newcommand\AKE{\mathrm{AKE}}
\renewcommand\square\blacksquare
\newcommand\AKEsquaretall\blacklozenge

\title[]{Elimination results for tame fields with finite residue fields}
\author{Sylvy Anscombe}
\author{Blaise Boissonneau}
\thanks{\today.}
\address{Universit\'{e} Paris Cité, Sorbonne Universit\'{e}, CNRS, IMJ-PRG, F-75013 Paris, France}
\email{sylvy.anscombe@imj-prg.fr}
\address{Heinrich Heine University Düsseldorf, Faculty of Mathematics and Natural Sciences, Universitätsstr.~1, 40225 Düsseldorf, Germany}
\email{blaise.boissonneau@hhu.de}

\begin{document}
\begin{abstract}
Building on work of Kuhlmann and Lisinski,
we study the theory of the Hahn series field $\Hs{\FF_{q}}{\QQ}$, over a finite field $\FF_{q}$, equipped with the $t$-adic valuation, in a language of valued fields.
We prove that every formula is equivalent to a formula $\exists y\colon f(x_{1},\ldots,x_{n},y)=0$, for a polynomial $f\in\ZZ[x_{1},\ldots,x_{n},y]$.
\end{abstract}
\maketitle

\section{Introduction}

We extend the known model theory of
tame and separably tame valued fields,
in the sense of Kuhlmann,
specifically those with a finite
residue field.
We obtain a relative quantifier elimination-style result,
complementing completeness and decidability results of Lisinski,
which yields a simpler description of the definable sets in such structures.
This applies in particular
to fields of Hahn series $\Hs{k}{\Gamma}$
for which the residue field $k$ is finite
and the value group $p$-divisible, where $p$ is the characteristic of $k$.

Kuhlmann's theory of tame and separably tame valued fields was developed in
\cite{Kuh16,KuhlmannPal,A-Lambda},
in which it was proved that separably tame valued fields satisfy a range of ``Ax--Kochen/Ershov principles''.
Extensions to this theory can also be found in
\cite{KR-dr},
who investigate ``semi-tame'' valued fields,
and in
\cite{JahnkeKartas},
who study ``roughly tame'' valued fields.
As an example, is it proved in \cite{Kuh16} that tame valued fields of equal characteristic form an $\AKE^{\equiv}$-class:
that is, two such valued fields $(K,v)$ and $(L,w)$ are elementarily equivalent
in the language $\Lval$ of valued fields
if and only if
their residue fields are elementarily equivalent
in the language of rings $\Lring$
and
their value groups are elementarily equivalent
in the language of ordered abelian groups $\Loag$.
On the other hand, it is clear---see~\cite{AK16}---%
that tame valued fields of mixed characteristic are not an $\AKE^{\equiv}$-class.
A deeper analysis of the mixed characteristic case can be found in the recent paper~\cite{KD-mixed}.

Lisinksi's work, including two papers\footnote{These papers appear to be unpublished at the time of writing, but are available in pre-print form.}~\cite{Lis,Lis_second} and his thesis~\cite{Lis_thesis},
contain the following theorem about theories of certain tame valued fields of Hahn series in the language $\Lvalpi$ of valued fields enriched by a constant symbol $\varpi$ for the formal variable.

\begin{theorem*}[{Lisinski, \cite[Theorem 1]{Lis}}]\label{thm:Lisinski}
	For every perfect field $k$ and $p$-divisible pointed value group $(\Gamma,\gamma)$,
	the $\Lvalpi$-theory
	$\Th(\Hs{k}{\Gamma},v_{t},t)$
	is axiomatized by the theory of equicharacteristic tame valued fields,
	with the theory $\Th(k)$ imposed on the residue field
	and the theory $\Th(\Gamma,\gamma)$ imposed on the pointed value group
	(i.e.~the value group with the value of the interpretation of $\varpi$ distinguished)
	and
	together with
	\begin{align*}
		S &=
		\{\exists y\;(f(\varpi,y)=0\wedge v(y)\geq0)\mid\text{$f\in\FF_{p}[x,y]$ such that $f(t,y)$ is monic in $y$ and has root in $\Hs{k}{\Gamma}$}\}.
	\end{align*}
	Consequently, 
	$\Th(\Hs{k}{\Gamma},v,t)$ is decidable
	if and only if
	$\Th(\Gamma,\gamma)$,
	$\Th(k)$,
	and $S$ are decidable.
\end{theorem*}

Combined with a result of Kedlaya from~\cite{Kedlaya}, this yields the following.

\begin{corollary*}[{Lisinski, \cite{Lis}}]
	For every prime power $q=p^{\ell}$,
	both
	$\Th(\Hs{\FF_{q}}{\QQ},v_{t},t)$
	and
	$\Th(\Hs{\FF_{q}}{p^{-\infty}\ZZ},v_{t},t)$
	are decidable.
\end{corollary*}

We prove a quantifier elimination-style result that strengthens Lisinski's Theorem
in the case of tame valued fields with a finite
residue field.
We view valued fields as structures in a three-sorted language $\Lval$ of valued fields,
defined in section~\ref{section:Preliminaries}.
We write $\Epos(\Lring)$ for the set of formulas in the language of rings of the form $\exists y\colon f(\underline{x},y)=0$ for a polynomial $f\in\ZZ[\underline{x},y]$.

\begin{introtheorem}\label{introthm:main}
	Let $q=p^{\ell}$ be a prime power.
	For every $\Lval$-formula
	$\varphi(\underline{x})$
	in the language $\Lval$ of valued fields
	there are 
	$\Lval$-formulas
	$\psi_{1}(\underline{x}),\ldots,\psi_{n}(\underline{x})$
	without quantifiers over
	the main field sort $\bbK$ or over the residue field sort $\bbk$, and
	$\Epos(\Lring)$-formulas
	$\chi_{1}(\underline{x}),\ldots,\chi_{n}(\underline{x})$
	(interpreted in the sort for the main field),
	such that 
	\begin{align*}
		(K,v)\models\forall\underline{x}\;\left(\varphi(\underline{x})\leftrightarrow\bigvee_{i=1}^{n}\psi_{i}(\underline{x})\wedge\chi_{i}(\underline{x})\right),
	\end{align*}
	for every tame valued field $(K,v)$ with residue field $\FF_{q}$.
	Moreover the map $\varphi\mapsto\bigvee_{i}\psi_{i}\wedge\chi_{i}$ is computable.
\end{introtheorem}

This theorem is proved after Theorem~\ref{thm:main}.
When the value group is divisible, 
we have
a stronger result:

\begin{introtheorem}[{Theorem~\ref{thm:FqQ}}]\label{introthm:FqQ}
	Let $q=p^{\ell}$ be a prime power.
	For every $\Lring$-formula $\varphi(\underline{x})$ 
	there is an $\Epos(\Lring)$-formula
	$\chi(\underline{x})$
	such that
	\begin{align*}
		\Hs{\FF_{q}}{\QQ}\models\forall\underline{x}\;(\varphi(\underline{x})\leftrightarrow\chi(\underline{x})).
	\end{align*}
	Moreover the map $\varphi\mapsto\chi$ is computable.
\end{introtheorem}

Our method is to apply a lemma of Kuhlmann,
to argue that relatively algebraically closed subfields of certain tame valued fields
are themselves tame valued fields --- this doesn't hold for arbitrary tame valued fields.
Thereafter we apply the embedding theorems for tame valued fields developed in \cite{Kuh16,KuhlmannPal,A-Lambda}.
Along the way we prove the following theorem on the definability of henselian valuation rings and valuation ideals (with finite residue fields) by existential $\Lring$-formulas of low complexity,
refining a result from~\cite{Feh15}.

\begin{introtheorem}[{Theorem~\ref{thm:E1}}]\label{introthm:E1}
For each prime power $q=p^{\ell}$,
there exists 
an $\Epos(\Lring)$-formula 
$\zeta_{q}(x)$
\resp{respectively, $\epsilon_{q}(x)$}
that defines the valuation ring
$\mathcal{O}_{v}$
\resp{respectively,
the valuation ideal
	$\mathfrak{m}_{v}$}
in all henselian valued fields
$(K,v)$ with residue field $\FF_{q}$.
\end{introtheorem}

% We deduce also the following result on model-theoretic algebraic independence.
% 
% \begin{introcorollary}\label{introcor:veryslim}
% 	Tame valued fields
% 	with finite residue field are very slim, in the sense of \cite{JK},
% 	i.e.~model-theoretic and field-theoretic algebraic closures coincide.
% \end{introcorollary}

\section{Preliminaries}
\label{section:Preliminaries}

Several languages are used in the literature to handle the model theory of certain imperfect fields.
For the sake of concreteness we recall the \define{lambda language} $\Llambda$,
as presented in \cite[\S2.3]{A-Lambda},
which has signature
$$\{l_{I}(x,\underline{y})\mid I\in p^{n},p\in\PP,n<\omega,\underline{y}=(y_{i})_{i<n},n<\omega\},$$
consisting of a family of function symbols which we interpret by the parameterized lambda maps.
There is an $\Llambdaring$-theory $T_{\lambda}$ that axiomatizes this interpretation, so that every field, construed as an $\Lring$-structure, expands uniquely to an $\Llambdaring$-structure that models $T_{\lambda}$.

\begin{remark}
We take care to prove our main results in the context of separably tame valued fields.
However, to help any reader uninterested in imperfect fields we indicate throughout the simplifications to be made for the perfect case.
Essentially, one may replace $\Llambdaring$ with $\Lring$.
\end{remark}

Let $\Lval$ be the usual three-sorted language of valued fields,
with sorts $\bbK$, $\bbk$, and $\bbG$,
where $\bbK$ and $\bbk$ are equipped with $\Lring$, and $\bbG$ is equipped with $\Loag$.
Moreover there is a function symbol $\bbval$ for the valuation, from $\bbK$ to $\bbG$,
and a function symbol $\bbres$ for the residue map, from $\bbK$ to $\bbk$.
Let $\ik$ and $\iG$ be the maps between formulas coming from the standard interpretations of $\bbk$ and $\bbG$ in $\Lval$,
see for example~\cite[Definition 3.15]{AF-fragments}.
Let $\Llambdaval$ be the expansion of $\Lval$ by $\Llambda$ on the sort $\bbK$, so that $\bbK$ is equipped with $\Llambdaval$.

In the interest of allowing our results to be stated {resplendently}, we
follow the notation used in \cite[\S4.2]{A-Lambda}:
a \define{$(\bbk,\bbG)$-expansion}
of $\Llambdaval$ 
\resp{in the perfect case: $\Lval$}
is
any expansion in which
\begin{itemize}
	\item
		the language on the residue field sort $\bbk$ is expanded to a language $\Lk\supseteq\Lring$,
	\item
		the language on the value group sort $\bbG$ is expanded to a language $\LG\supseteq\Loag$, and
        \item
		nothing else is added.
\end{itemize}
We denote such an expansion of $\Llambdaval$ by $\Llambdaval(\Lk,\LG)$
\resp{in the perfect case: such an expansion of $\Lval$ is denoted $\Lval(\Lk,\LG)$}.

Our results are concerned with gaining a tight control on the syntatic complexity of formulas in certain theories.
To that end, for $\Lang=\Lring,\Llambdaring$ we let
\begin{itemize}
	\item
		$\Epos(\Lang)$ be the set of $\Lang$-formulas of the form
		$\exists y\colon\varphi(\underline{x},y)$,
		where
		$y$ is a single variable
		and
		$\varphi$ is an equality of $\Lang$-terms\footnote{We may allow $\varphi$ also to be either $\top$ or $\bot$, but these are equivalent to $0=0$ and $0=1$ (respectively), modulo the theory of fields, so in this setting this choice is immaterial.}.
\end{itemize}
In case $\Lang=\Lring$,
	such equalities of terms are (up to computable equivalence) of the form $f(\underline{x},y)=0$, where $f\in\ZZ[\underline{x},y]$.
In case $\Lang=\Llambdaring$,
	such equalities of terms are a little more complicated
	but are (up to computable equivalence) of the form
	$\exists y\colon t(\underline{x},y)=0$, where $t(\underline{x},y)$ is an $\Llambdaring$-term;
	these terms are studied in greater detail in \cite[\S2]{A-Lambda} and \cite[\S2]{SotoMoreno}.

For a first-order language $\Lang$,
we recall the notion of an \define{$\Lang$-fragment},
as defined in \cite[\S2]{AF-AE}:
it is any set of $\Lang$-formulas that includes $\top$ and $\bot$, which is closed under (finite) conjunctions and disjunctions,
and which is closed under substitution of free variables.
We observe that neither is $\Epos(\Lring)$ an $\Lring$-fragment nor is $\Epos(\Llambdaring)$ an $\Llambdaring$-fragment, since neither is closed under conjunction or disjunction.
We let
\begin{itemize}
\item
	$\FEpos(\Lang)$ be the $\Lang$-fragment generated by $\Epos(\Lang)$.
\end{itemize}

\noindent
Given a $(\bbk,\bbG)$-expansion $\Lang=\Llambdaval(\Lk,\LG)$ of $\Llambdaval$, we let

	\begin{enumerate}[{\bf(i)}]
			\setcounter{enumi}{0}

	\item
		$\Fk(\Lang)$
		be the $\Lang$-fragment of formulas
		$\rho(\underline{s}(\underline{x},\underline{r}))$
		where $\rho$ is an $\Lk$-formula and $\underline{s}$ is a tuple of terms $\bbK\times\bbk\rightarrow\bbk$,

	\item
		$\FG(\Lang)$
		be the $\Lang$-fragment of formulas
		$\gamma(\underline{t}(\underline{x},\underline{\alpha}))$
		where $\gamma$ is an $\LG$-formula and $\underline{t}$ is a tuple of terms $\bbK\times\bbG\rightarrow\bbG$,
	\item
		$\FGal(\Lang)$ be the $\Lang$-fragment generated by
		$\Fk(\Lang)\cup\FG(\Lang)\cup\FEone(\Lring)$,
		and
	\item
		$\FlGal(\Lang)$ be the $\Lang$-fragment generated by
		$\Fk(\Lang)\cup\FG(\Lang)\cup\FEone(\Llambdaring)\cup X$, where $X$ is the set of sentences that axiomatize various finite elementary imperfection degrees.
\end{enumerate}
We call $\FGal(\Lang)$ the \define{Galois fragment} of $\Lang$
and
we call $\FlGal(\Lang)$ the \define{$\lambda$-Galois fragment} of $\Lang$.

\begin{remark}\label{rem:basic}
	It is well known that, given a polynomial $f=\sum_{i\leq n}c_{i}X^{i}$ of degree $n\geq1$, there is a computable map
$\FEpos(\Lring)\rightarrow\Epos(\Lring)$
such that
$\varphi$ and $\epsilon\varphi$ are equivalent modulo the theory of fields in which $f$ has no root.
Concretely, 
modulo this theory,
 	the disjunction
	$\exists y_{1}\;g_{1}(\underline{x},y_{1})=0\vee\exists y_{2}\;g_{2}(\underline{x},y_{2})=0$
	is equivalent to $\exists z\;g_{1}(\underline{x},z)g_{2}(\underline{x},z)=0$,
	and the conjunction
	$\exists y_{1}\;g_{1}(\underline{x},y_{1})=0\wedge\exists y_{2}\;g_{2}(\underline{x},y_{2})=0$
	is equivalent to $\exists z\;h(\underline{x},z)=0$,
	where $h(\underline{x},z)=\sum_{i\leq n}c_{i}g_{1}(\underline{x},z)^{i}g_{2}(\underline{x},z)^{n-i}$.
\end{remark}

\begin{remark}\label{def:En}
Later in section~\ref{section:Diophantine} we will briefly use the following system of notation:
	Given any set $F\subseteq\Form(\Lang)$ and $n\in\NN$,
	we denote by $[\exists_{n}]F$ the set of formulas $\exists y_{1}\ldots\exists y_{n}\;\varphi$, for $\varphi\in F$.
\end{remark}

\section{Diophantine henselian valuation rings and maximal ideals --- revisited}
\label{section:Diophantine}

Let $\H$ denote the $\Lval$-theory of henselian valued fields.
For any $\Lring$-theory $R$, let $\H(R)=\H\cup\iota_{\bbk}R$ denote the theory $\H$ together with axioms that assert the residue field is a model of $R$.

In this section we study the definability of $\mathcal{O}_{v}$ [respectively $\mathfrak{m}_{v}$] in $K$ by existential $\Lring$-formulas $\chi(x)$, without parameters, for $(K,v)\models\H$.
We are concerned both by issues of
\define{uniformity},
i.e.~for which classes $C$ of models of $\H(R)$, for a given $\Lring$-theory $R$,
there is one such formula $\chi(x)$ that defines $\mathcal{O}_{v}$ [respectively, $\mathfrak{m}_{v}$] in $K$ for all $(K,v)\in C$;
and by issues of 
\define{complexity},
i.e.~for which $\Lring$-fragments
$F$ (inside the fragment of existential formulas)
may we find such a formula $\chi(x)$ in $F$. These questions have a rich history and are a very active topic of research, we refer to \cite{FJ-survey} for a survey of recent, albeit by now a decade old, results in this area.

Focusing our attention to equal positive characteristic, the definability of $\ips{\FF_{q}}$ in $\ps{\FF_{q}}$ by an existential $\Lring$-formula (without parameters) was established in~\cite{AK14}, then dramatically simplified and extended to more general residue fields in~\cite{Feh15}.

\begin{theorem}[{\cite[Theorems 2.6 and 3.5]{Feh15}}]\label{thm:Fehm}
For a field $k$ that is either finite or is PAC and does not contain the algebraic closure of its prime subfield,
there is an existential $\Lring$-formula $\eta_{f}(x)$ that defines $\mathcal{O}_{v}$ in $K$ for all $(K,v)\models\H(\Th(k))$.
\end{theorem}

Fehm's results also give existential definitions uniform over various families of finite residue fields.
This perspective was generalized in \cite{AF17}, which isolated conditions on the residue field that characterize the existence of an existential definition of $\mathcal{O}_{v}$ \resp{respectively $\mathfrak{m}_{v}$} in some $(K,v)\models\Hen(\Th(k))$, equivalently uniformly across all $(K,v)\models\H(\Th(k))$.

Our present aim is to find $\Epos(\Lring)$-formulas to define $\mathcal{O}_{v}$
\resp{respectively $\mathfrak{m}_{v}$}
uniformly in the theory of henselian valued fields of equal characteristic with residue field $\FF_{q}$.
The formulas constructed in \cite{Feh15} bear closer examination.
We will make frequent use of the Artin--Schreier polynomial $\wp_{q}(X)=X^{q}-X$.

\begin{lemma}[{Reformulation of~\cite[Lemma 2.1]{Feh15}}]\label{lem:Fehm}
Let $(K,v)$ be a valued field
and
let $f\in\mathcal{O}_{v}[Y]$ be any monic polynomial for which the residue $\bar{f}$ has no zero in the residue field $Kv$.
Then for $y\in K$ we have
\begin{align*}
f(y)
&\in
\left\{
\begin{array}{lll}
K\setminus\mathcal{O}_{v}
&\text{if }v(y)<0\text{ and}\\
\res_{v}^{-1}(\bar{f}(Kv))\subseteq\mathcal{O}_{v}^{\times}
&\text{if }v(y)\geq 0.\\
\end{array}
\right.
\end{align*}
Now let $a\in K$
and set
$
U_{f,a}:=f(K)^{-1}-f(a)^{-1}
$.
We have
\begin{enumerate}[{\bf(i)}]
\item
$f(K)^{-1}\subseteq\mathcal{O}_{v}$
\item
$U_{f,a}\subseteq\mathcal{O}_{v}$
\item
	If $a\in\mathcal{O}_{v}$ and $f'(a)\notin\mathfrak{m}_{v}$, and $(K,v)$ is henselian, then $\mathfrak{m}_{v}\subseteq U_{f,a}$
\end{enumerate}
\end{lemma}

It is clear that $U_{f,a}$ is definable by an $\Epos(\Lring)$-formula
$\varphi_{f,a}(x)\equiv\exists y\colon f(y)f(a)x=f(a)-f(y)$,
which however includes parameters coming from the coefficients of $f$ and from $a$.
Fehm also considered $U_{f}:=f(K)^{-1}-f(K)^{-1}$,
which is definable by an $\exists_{2}(\Lring)$-formula $\varphi_{f}(x)\equiv\exists y,z\colon f(y)f(z)x=f(z)-f(y)$, with parameters now only coming from the coefficients of $f$.
In the case that
$k=\FF_{q}$ is finite,
Fehm introduced the $[\exists_{2}]\FEpos(\Lring)$-formula
$\eta_{f}(x)\equiv\exists r,s\colon (x=r+s\wedge\varphi_{f}(r)\wedge\wp_{q}(s)=0)$,
using $f\in\FF_{p}[X]$ from \cite[Lemma 2.4]{Feh15},
which then defines $\mathcal{O}_{v}$ without additional parameters, in all $(K,v)\models\H(\Th(\FF_{q}))$.

% Moreover, the formula given in~\cite{Feh15} defines the valuation ring uniformly in any model of $\H(\Th(k))$, where $k$ is a given PAC field not containing an algebraically closed subfield.
% Later in~\cite{AF17} a comprehensive analysis of the problem described an elementary property of residue fields that characterizes the non-definability of valuation rings by existential $\Lring$-formulas without parameters.

\begin{theorem}\label{thm:E1}
For each prime power $q=p^{\ell}$,
there exists a formula
	$\zeta_{q}(x)$
\resp{respectively
	$\epsilon_{q}(x)$}
in $\Epos(\Lring)$
that defines~%
$\mathcal{O}_{v}$
\resp{respectively
	$\mathfrak{m}_{v}$},
in all $(K,v)\models\H(\Th(\FF_{q}))$.
\end{theorem}
\begin{proof}
By~\cite[Lemma 2.4]{Feh15}, for example, we may choose monic $f\in\ZZ[X]$ of which the reduction $\bar{f}$ modulo $p$ has no zero in $\FF_{q}$ and such that $\bar{f}'(0)$ is not zero.
Observe that the images of $f$ and $a=0$ in $\mathcal{O}_{v}[X]$ and $\mathcal{O}_{v}$ satisfy the hypotheses of Lemma~\ref{lem:Fehm} in any $(K,v)\models\H(\Th(\FF_{q}))$.
We consider the
$\Epos(\Lring)$-formula
$\varphi_{f,0}(x)\equiv\exists y\colon f(y)f(0)x=f(0)-f(y)$.
In any model $(K,v)\models\H(\Th(\FF_{q}))$,
$\varphi_{f,0}(x)$ defines $U_{f,0}$.
Moreover, since the valuation ring $\mathcal{O}_{v}$ is the pre-image of $U_{f,0}$ under $\wp_{q}$ by henselianity,
it is defined by
\begin{align*}\zeta_{q}(x)\equiv\varphi_{f,0}(\wp_{q}(x)).\end{align*}
Similarly, the valuation ideal $\mathfrak{m}_{v}$ is the image of $U_{f,0}$ under $\wp_{q}$,
and so is defined by
\begin{align*}\epsilon_{q}(x)\equiv\exists y\colon (f(y)f(0))^{q}x=(f(0)^{q}-f(y))^{q}-(f(0)-f(y))(f(y)f(0))^{q-1}.\end{align*}
Both
$\zeta_{q}(x)$ and $\epsilon_{q}(x)$ are
$\Epos(\Lring)$-formulas.
\end{proof}

This proves Theorem~\ref{introthm:E1} from the introduction.

\subsection{Digression: definability by rational functions}

\begin{remark}\label{rem:rational_functions}
The rational function
$\beta_{f}(y)=(f(0)-f(y))/f(y)f(0)$
plays a hidden role in the preceding proof.
Briefly we consider the language $\Lfield$ of fields, by which we mean $\Lring$ expanded by a unary function symbol ${\cdot}^{-1}$. We view any field as an $\Lfield$-structure by letting $x^{-1}$ be the multiplicative inverse for nonzero $x$, and writing $0^{-1}=0$.
Then $\Epos(\Lfield)$ is the set of formulas $\exists y\;f(\underline{x},y)=0$ for a rational function $f\in\ZZ(\underline{X},Y)$.
The $\Epos(\Lfield)$-formula
$\varphi'(x)\equiv\exists y\colon x=\beta_{f}(y)$
is equivalent to $\varphi_{f,0}(x)$ modulo the theory $T_{f}$ of fields in which $f$ has no roots.
Similarly,
the formula
$\zeta'(x)\equiv\exists y\colon\wp_{q}(x)=\beta_{f}(y)$
is equivalent to $\zeta_{q}(x)$,
and $\epsilon'(x)\equiv\exists y\colon x=\wp_{q}(\beta_{f}(y))$
is equivalent to $\epsilon_{q}(x)$,
modulo $T_{f}$.

We observe that each of $\varphi'(x)$, $\zeta'(x)$, and $\epsilon'(x)$ is
\define{separated},
i.e.~of the form
$\exists y\colon g(\underline{x})=h(y)$,
for a polynomial $g\in\FF_{p}[\underline{X}]$ and a rational function $h\in\FF_{p}(Y)$,
Moreover, the formulas $\varphi'(x)$ and $\epsilon(x)$ are even
\define{rational},
i.e.~of the form
$\exists y\colon x=h(y)$,
for a rational function $h\in\FF_{p}(Y)$.
\end{remark}

\begin{example}
For a more concrete example, we may take $f=\wp_{q}-1$ and $a=0$ to get the rational function
$\beta_{f}(Y)=(\wp_{q}(Y)-1)^{-1}-1$.
The separated $\Epos(\Lfield)$-formula
$\zeta'(x)$ is
	equivalent modulo $T_{f}$ to
\begin{align*}
&\exists y\colon \wp_{q}(x)=\frac{1}{\wp_{q}(y)-1}-1,
\end{align*}
and the rational $\Epos(\Lfield)$-formula
$\epsilon'(x)$ is (equivalent to)
\begin{align*}
&\exists y\colon x=\wp_{q}\left(\frac{1}{\wp_{q}(y)-1}\right).
\end{align*}
Of course, $\zeta'(x)$ and $\epsilon'(x)$
	define $\mathcal{O}_{v}$ and $\mathfrak{m}_{v}$, respectively, uniformly in all
$(K,v)\models\H(\Th(\FF_{q}))$.
\end{example}

For a prime power $q=p^{\ell}$,
we consider the rational function
$\alpha_{q}(U,Y)=(\wp_{q}(Y)-U)^{-1}\in\FF_{p}(U,Y)$.

\begin{lemma}\label{lem:rational_function}
Let $(K,v)$ be a valued field with $Kv\cong\FF_{q}$,
and let $u\in\mathcal{O}_{v}^{\times}$.
Since $\wp_{q}$ vanishes on $\FF_{q}$, $\alpha_{q}(u,Y)$ well-defines a function $K\rightarrow K$.
More precisely, we have
\begin{align*}
\alpha_{q}(u,y)
&\in
\left\{
\begin{array}{lll}
	\mathfrak{m}_{v}
&\text{if }v(y)<0\text{ and}\\
	-u^{-1}+\mathfrak{m}_{v}
&\text{if }v(y)\geq 0,\\
\end{array}
\right.
\end{align*}
for $y\in K$.
\end{lemma}

For each $u\in\mathcal{O}^{\times}_{v}$, the lemma implies that the rational function $\alpha_{q}(u,Y)$ restricts to a function
$K\setminus\mathcal{O}_{v}\rightarrow\mathfrak{m}_{v}$
and also to a function
$\mathcal{O}\rightarrow -u^{-1}+\mathfrak{m}_{v}$.
The former is not in general surjective:
for example neither $0$ nor $t$
is in the image of $\alpha_{q}(u,X)$ when $K=\ps{\FF_{q}}$.

\begin{lemma}\label{lem:rational_function_2}
Let $(K,v)\models\H(\Th(\FF_{q}))$
and let $u\in\mathcal{O}_{v}^{\times}$.
Then
	$\alpha_{q}(u,w+\mathfrak{m}_{v})=-u^{-1}+\mathfrak{m}_{v}$,
for each $w\in\mathcal{O}_{v}$.
\end{lemma}
\begin{proof}
The left-to-right inclusion is clear.
	Let $x\in\mathfrak{m}_{v}$.
	We seek $y\in\mathfrak{m}_{v}$ such that $\alpha_{q}(u,w+y)=-u^{-1}+x$.
	Equivalently we aim to find a root of the polynomial $f(Y)=\wp(w+Y)-u-(-u^{-1}+x)^{-1}\in\mathcal{O}_{v}[Y]$ in $\mathfrak{m}_{v}$.
	The polynomial is monic, $f(0)\in\mathfrak{m}_{v}$, and $f'(0)=-1\notin\mathfrak{m}_{v}$.
	Thus by henselianity of $(K,v)$ there exists $y\in\mathfrak{m}_{v}$ such that $f(y)=0$ as required.
\end{proof}

\begin{proposition}\label{prp:rational_function_NOT2}
Suppose $p\neq2$.
Let $(K,v)\models\H(\Th(\FF_{q}))$
and
let $u\in\mathcal{O}_{v}^{\times}$.
We have
\begin{enumerate}[{\bf(i)}]
\item
	$\{\alpha_{q}(u,w+y)+\alpha_{q}(-u,w+y)\mid y\in\mathfrak{m}_{v}\}=\mathfrak{m}_{v}$, for each $w\in\mathcal{O}_{v}$.
\end{enumerate}
It follows that
\begin{enumerate}[{\bf(i)}]
\setcounter{enumi}{1}
\item
	$\exists y\colon x=\alpha_{q}(u,y)+\alpha_{q}(-u,y)$ is a rational $\Epos(\Lfield)$-formula defining $\mathfrak{m}_{v}$.
\end{enumerate}
\end{proposition}
\begin{proof}
Let $(K,v)\models\H(\Th(\FF_{q}))$
and 
let $u\in\mathcal{O}_{v}^{\times}$.
With the aim of proving {\bf(i)} we let
	$x\in\mathfrak{m}_{v}$.
	We will find $y\in\mathfrak{m}_{v}$ such that
$x=\alpha_{q}(u,w+y)+\alpha_{q}(-u,w+y)$.
	Since every element of $\mathfrak{m}_{v}$ admits a pre-image in $w+\mathfrak{m}_{v}$ under the polynomial map $\wp_{q}$%
	, by henselianity, it suffices to find $z\in\mathfrak{m}_{v}$ such that
$2z/(z^{2}-u^{2})=x$.
By henselianity of $(K,v)$,
and since $p\neq2$,
we have
	$1+x^{2}\mathcal{O}_{v}\subseteq(1+x\mathfrak{m}_{v})^{2}$.
Therefore 
$1+u^{2}x^{2}=(1+xs)^{2}$,
	for some $s\in\mathfrak{m}_{v}$,
and so the polynomial
$xZ^{2}-Z-u^{2}x$ has roots in $K$:
either a single root $z=0$, if $x=0$, or
otherwise
two roots
\begin{align*}
z_{1}:=\frac{1+(1+xs)}{2x}
\end{align*}
and
\begin{align*}
	z_{2}:=\frac{1-(1+xs)}{2x}=-\frac{s}{2}\in\mathfrak{m}_{v}.
\end{align*}
As mentioned above,
	by henselianity there is $y\in\mathfrak{m}_{v}$ such that $\wp(w+y)=z_{2}$, which proves {\bf(i)}.
Claim {\bf(ii)} follows by combining~{\bf(i)} with Lemma~\ref{lem:rational_function}.
\end{proof}

\begin{observation}[The Kochen operator]\label{obs:Kochen}
Let $p\neq2$ and take $u\in\ZZ$ not divisible by $p$.
Consider the rational function
$\gamma_{q}^{[u]}(Y):=
\alpha_{q}(u,Y)+\alpha_{q}(-u,Y)\in\ZZ(Y)$
appearing in the previous proof.
By taking an average of $\alpha_{q}(u,Y)$ and $\alpha_{q}(-u,Y)$, and dividing by $p$, we recover the usual $p$-Kochen operator:
\begin{align*}
\frac{1}{2p}\gamma_{q}^{[u]}(Y)
=
\frac{1}{2p}\left(\alpha_{q}(u,Y)+\alpha_{q}(-u,Y)\right)
=
\frac{1}{2p}\left(\frac{1}{\wp_{q}(Y)-u}+\frac{1}{\wp_{q}(Y)+u}\right)
=
\frac{1}{p}\frac{\wp_{q}(Y)}{\wp_{q}(Y)^{2}-u^{2}},
\end{align*}
the image of which is the valuation ring in every unramified henselian $p$-valued field of mixed characteristic.
However we observe that the image of $\gamma_{q}^{[u]}(Y)$ is already the maximal ideal in every model of $\H(\Th(\FF_{q}))$ regardless of ramification and even in equal characteristic $p$.
\end{observation}

\begin{question}
For $(K,v)\models\H(\Th(\FF_{q}))$, by Theorem~\ref{thm:E1} and Remark~\ref{rem:rational_functions}, the valuation ring $\mathcal{O}_{v}$ is definable by a separated $\Epos(\Lfield)$-formula.
Is it definable by a rational $\Epos(\Lfield)$-formula?
\end{question}

\begin{question}
For $k$ PAC and not containing the algebraic closure of its prime field, and $(K,v)\models\H(\Th(k))$, we know by~\cite{Feh15} that $\mathcal{O}_{v}$ is definable by an $\exists$-formula.
By examining the proof of \cite[Theorem~3.5]{Feh15}, it may be seen immediately to be $\exists_{2}\exists_{4}$-definable.
Is $\mathcal{O}_{v}$ definable by a formula that is $\exists_{2}$, or $\exists_{1}$, or even separated?
\end{question}

\section{Elimination to the Galois fragment}
\label{section:Elimination}

Let $F/C$ be a field extension.
We define two related notions of imperfection degree:
the \define{(relative) imperfection degree} of $F$ over $C$, denoted $\impdeg(F/C)$,
is the cardinality of a $p$-basis of $F$ over $C$,
and the \define{absolute imperfection degree} of $F$, denoted $\impdeg(F)$, is $\impdeg(F/\FF)$.
The \define{elementary (relative) imperfection degree}
(also called ``Ershov degree/invariant'')
of $F$ over $C$
is defined to be $\Impdeg(F/C)=\impdeg(F/C)$ if $\impdeg(F)$ is finite;
or $\Impdeg(F/C)=\infty$ if $\impdeg(F)$ is infinite.
In this paper we will always be in the setting that $F/C$ is separable,
in which case $\impdeg(F)=\impdeg(F/C)+\impdeg(C)$ and $\Impdeg(F)=\Impdeg(F/C)+\Impdeg(C)$, writing $\infty+n=n+\infty=\infty$ for all $n\in\NN\cup\{\infty\}$.

Following \cite{Kuh16},
a valued field $(K,v)$ is \define{separably tame}
if it is henselian and separably defectless, has $p$-divisible value group, and perfect residue field,
where $p$ is the residue characteristic exponent,
i.e.~$p=1$ when $\ch(Kv)=0$, and $p=\ch(Kv)$ otherwise.
Let $\STVF$
be
the $\Llambdaval$-theory of
separably tame valued fields
(in particular extending $T_{\lambda}$).
For a theory $R$ of fields
and a theory $G$ of pointed ordered abelian groups,
we denote by
$\STVF(R,G)$
the $\Lval$-theory
$\STVF\cup\iota_{\bbk}R\cup\iota_{\bbG}G$.
For $I$ a finite or cofinite subset
of $\NN\cup\{\infty\}$
we let $\STVF_{I}$ be the extension of $\STVF$ by $\Lring$-axioms that assert the elementary imperfection degree of the field in the main sort is in $I$.
We also will write $\STVF_{I}(R,G)=\STVF_{I}\cup\iota_{\bbk}R\cup\iota_{\bbG}G$.

The following theorem states the ``subcompleteness'' form of the Ax--Kochen/Ershov principles for separably tame valued fields.

\begin{theorem}[{\cite{Kuh16} for $\frakI=0$; \cite{KuhlmannPal} for $\frakI<\infty$; \cite{A-Lambda} for $\frakI=\infty$}]\label{thm:subcompleteness}
	Let $(K_{1},v_{1}),(K_{2},v_{2})\models\STVF$
	have common
	separably defectless subfield $(K,v)$,
	with $K_{1}/K$ regular%
	\footnote{A field extension $L/K$ is \define{regular} if $L/K$ is linearly disjoint from the algebraic closure of $K$.}%
	.
	Then
	$(K_{1},v_{1})\equiv_{K}(K_{2},v_{2})$
	if and only if
	\begin{enumerate}[{\bf(i)}]
	\item
	$k_{v_{1}}\equiv_{k_{v}}k_{v_{2}}$,
	\item
	$\Gamma_{v_{1}}\equiv_{\Gamma_{v}}\Gamma_{v_{2}}$,
	and
	\item
	$\Impdeg(K_{1}/K)=\Impdeg(K_{2}/K)$.
	\end{enumerate}
\end{theorem}

One of the key tools in the proofs of Theorem~\ref{thm:subcompleteness} is the following lemma.

\begin{lemma}[{Going down, \cite[Lemma 3.7]{Kuh16}}]\label{lem:going_down}
Let $(L,v)$ be a
separably tame valued field
and
let $(K,v)\subseteq(L,v)$ be such that $L/K$ is regular.
Suppose that $Lv/Kv$ is algebraic.
Then
$(K,v)$ is separably tame,
$Lv=Kv$,
and
$vL/vK$ is torsion-free.
\end{lemma}

\begin{corollary}\label{cor:going_down}
	Let $(L,v)$ be a separably tame valued field with algebraic residue field $Lv$
	(i.e.~$Lv$ is an algebraic extension of its prime subfield)
and let $K\subseteq L$ be any subfield with $L/K$ regular.
	Then $v$ induces on $K$ a valuation $v'$ such that $(K,v')$ is separably tame, $Lv=Kv'$, and $vL/v'K$ is torsion-free.
\end{corollary}

\begin{remark}
Before we can state the proposition, we recall several elementary definitions from computability theory:
for $A,B\subseteq\NN$,
$A$ is \define{Turing reducible} to $B$,
written $A\Tred B$,
if there is a computable function $f:\NN\rightarrow\NN$ such that $A=f^{-1}(B)$.
Of course then $A$ and $B$ are \define{Turing equivalent}, written
$A\Teq B$,
if
both
$A\Tred B$ and $B\Tred A$.
The \define{Turing sum}
$A\Toplus B$ is then a code for the disjoint union of $A$ and $B$,
for example we may take $A\Toplus B=2A\cup(2B+1)$.
For more details on these foundational concepts, we refer to
\cite[\S11]{Post} and \cite[Definition 1.6.8]{Soare}.
When discussing computability of theories, we will always be working in languages that are \define{presented},
that is, the set of symbols in each language $\Lang$ is assumed to be equipped with an injection $\Lang\rightarrow\NN$ as described in \cite[Appendix A]{AF-AE}.
By Gödel coding, this injection extends to the formulas of $\Lang$.
We say that a theory is computable if its image under the latter injection is,
and the Turing sum of theories is defined to be the Turing sum of their images.
\end{remark}

For convenience we write $\FF_{0}=\QQ$.
	This proposition should be compared with
	\cite[Theorems 1.1, 1.2]{Lis},

Given a field extension $K/C$, the \define{regular closure} of $C$ in $K$, is the smallest subfield $D\subseteq K$ that contains $C$ such that $K/D$ is regular.
This coincides with the relative algebraic closure in $K$ of the Lambda closure $\Lambda_{K}C$, as described in \cite{DM,A-Lambda}.
In both Proposition~\ref{prp:mono} and Theorem~\ref{thm:main}
we denote by $\FF_{p}(\underline{a})^{\reg}$ the regular closure of $\FF_{p}(\underline{a})$ in $K$, which is understood to be equipped with the restriction of $v$.

\begin{proposition}[{Main Proposition}]\label{prp:mono}
Let $\Lang=\Llambdaval(\Lk,\LG)$ be a $(\bbk,\bbG)$-expansion of $\Llambdaval$.
	Let $(K,v),(L,w)\models\STVF\cup\iota_{\bbk}\Th(k)$,
	with $\underline{a}\in K$ and $\underline{b}\in L$ both $n$-tuples such that the $i$-th element of $\underline{a}$ and of $\underline{b}$ are in the same sort, for each $i$.
	Suppose that
\begin{itemize}
\item[{\Alg}]
	$\FF_{p}(\underline{a})^{\reg}$ is separably defectless.
\end{itemize}
	Then
$
		\tp_{\lGal}^{K}(\underline{a})\subseteq\tp_{\lGal}^{L}(\underline{b})
		\implies
		\tp^{K}(\underline{a})=\tp^{L}(\underline{b})
$.
% Consequently 
% \begin{enumerate}[{\bf(i)}]
	% \item
		% there is a computable function
		% $\epsilon:\Form(\Lang)[x]\rightarrow\FlGal(\Lang)[x]$
		% such that
		% $$\STVF(\Th^{\Lk}(k),\emptyset)\models\forall x\;(\varphi(x)\leftrightarrow\epsilon\varphi(x)),$$
		% for every $\varphi(x)\in\Form(\Lang)[x]$.
% 	\item
% 		$\Th(K,v,\underline{a})\equiv\STVF_{\frakI}(\Th(k_{v}),\Th(\Gamma_{v}))\cup\Th_{\Epos(\Llambdaring)}(K,\underline{a})$,
% 		where $\frakI=\Impdeg(K)$;
% 		and
% 	\item
% 		$\Th(K,v,\underline{a})\Deq\Th(k_{v})\Doplus\Th(\Gamma_{v})\Doplus\Th_{\Epos(\Llambdaring)}(K,\underline{a})$.
%	\end{enumerate}
\end{proposition}
\begin{proof}
	Let $A$ be the $\Lang$-substructure of $(K,v)$ generated by $\underline{a}$,
	with underlying field $K_{0}$,
	and let $B$ be the $\Lang$-substructure of $(L,w)$ generateted by $\underline{b}$,
	with underlying field $L_{0}$.
	We keep in mind, however, that the sorts $\bbk$ and $\bbG$ of $A$ may be bigger than the image of $K_{0}$ under the residue and valuation maps, respectively; the same holds for $B$ and $L_{0}$.
	Let $v_{0}$ and $w_{0}$ denote the valuations on $K_{0}$ and $L_{0}$ induced by $v$ and $w$, respectively.
	The hypothesis of the implication implies
	that there is an isomorphism $\varphi_{0}:A\rightarrow B$ which in particular is an isomorphism of valued fields $(K_{0},v)\rightarrow(L,v)$ that maps $\underline{a}\mapsto\underline{b}$.
	We observe that $K/K_{0}$ and $L/L_{0}$ are separable because $A$ and $B$ are $\Lang$-substructures of $K$ and $L$, respectively, and $\Lang$ includes the parametrized lambda functions.
	By the universal property of henselizations, $\varphi_{0}$ extends to a map $\varphi_{1}:A_{1}\rightarrow B_{1}$ where $A_{1}$ is the henselization of $A$ (taken inside $K$), and $B_{1}$ is the henselization of $B$ (taken inside $L$).
	We must continue to keep in mind that $A_{1}$ is in fact the $\Lang$-substructure of $K$ which is generated by $\underline{a}$ and the henselization of $K_{0}$, and the same remark holds for $B_{1}$.
	The assumption on types implies that $\varphi$ may be extended to an embedding from the relative algebraic closure $A_{2}$ of $A_{1}$ in $K$ into the relative algebraic closure $B_{2}$ of $B_{1}$ in $L$,
	where both of these algebraic closures are to be understood in the field-theoretic sense, so that the sort $\bbK$ of $A_{2}$ is exactly the field-theoretic relative algebraic closure of the sort $\bbK$ of $A_{1}$ in $K$, and similarly for $B_{2}$.
	Next, we may identify $A_{1}$ with $B_{1}$ along $\varphi_{1}$, in particular identifying $\underline{a}$ with $\underline{b}$.
	Let $K_{1}$ be the sort $\bbK$ of $A_{1}$.
	The hypothesis also implies that
	$vK\equiv_{vK_{1}}wL$,
	$Kv\equiv_{K_{1}v}Lw$,
	and
	$\Impdeg(K/K_{1})=\Impdeg(L/K_{1})$.
	By the hypothesis \Alg, $(K_{1},v_{1})$ is separably defectless,
	and so by Theorem~\ref{thm:subcompleteness} we have $(K,v)\equiv_{K_{1}}(L,w)$,
	i.e.~$\tp^{K}(\underline{a})=\tp^{L}(\underline{b})$.
\end{proof}

\begin{theorem}[{Main Theorem}]\label{thm:main}
Let $\Lang=\Llambdaval(\Lk,\LG)$ be a $(\bbk,\bbG)$-expansion of $\Llambdaval$,
and let $k$ denote an $\Lk$-structure expanding a finite
field.
Complete $\Lang$-types in the theory $\STVF(\Th^{\Lk}(k),\emptyset)$ are determined by their $\lambda$-Galois fragment.
Consequently,
there is a computable function
$\epsilon_{k}:\Form(\Lang)[\underline{x},\underline{r},\underline{\alpha}]\rightarrow\FlGal(\Lang)[\underline{x},\underline{r},\underline{\alpha}]$
such that
	\begin{align*}
		\STVF(\Th^{\Lk}(k),\emptyset)\models\forall\underline{x}\;(\varphi(\underline{x},\underline{r},\underline{\alpha})\leftrightarrow\epsilon_{k}\varphi(\underline{x},\underline{r},\underline{\alpha}))
	\end{align*}
		for every $\varphi(\underline{x},\underline{r},\underline{\alpha})\in\Form(\Lang)[\underline{x},\underline{r},\underline{\alpha}]$.
\end{theorem}
\begin{proof}
	The hypothesis \Alg\ of Propositionn~\ref{prp:mono} holds by Corollary~\ref{cor:going_down}.
	The existence of $\epsilon_{k}$ then follows from the conclusion of Proposition~\ref{prp:mono} by a standard argument using the separation lemma,
	and $\epsilon_{k}$ is computable by enumerating proofs from $\STVF(\Th^{\Lk}(k),\emptyset)$.
	For a detailed exposition,
	see \cite{AF-fragments}.
\end{proof}

We denote by $\TVF$ the $\Lval$-reduct of $\STVF_{0}$, which is exactly the theory of tame valued fields, i.e.~the theory of perfect separably tame valued fields.

\begin{theorem}[{Main Theorem---perfect case}]\label{thm:main_perfect}
Let $\Lang=\Lval(\Lk,\LG)$ be a $(\bbk,\bbG)$-expansion of $\Lval$,
and let $k$ denote an $\Lk$-structure expanding a finite
field.
Complete $\Lang$-types in the theory $\TVF(\Th^{\Lk}(k),\emptyset)$ are determined by their Galois fragment.
Consequently,
there is a computable function
$\epsilon_{k}:\Form(\Lang)[\underline{x},\underline{r},\underline{\alpha}]\rightarrow\FGal(\Lang)[\underline{x},\underline{r},\underline{\alpha}]$
such that
	\begin{align*}
		\TVF(\Th^{\Lk}(k),\emptyset)\models\forall\underline{x}\;(\varphi(\underline{x},\underline{r},\underline{\alpha})\leftrightarrow\epsilon_{k}\varphi(\underline{x},\underline{r},\underline{\alpha}))
	\end{align*}
		for every $\varphi(\underline{x},\underline{r},\underline{\alpha})\in\Form(\Lang)[\underline{x},\underline{r},\underline{\alpha}]$.
\end{theorem}
\begin{proof}
	The proof of Theorem~\ref{thm:main} works here, simply changing $\STVF$ for $\TVF$.
\end{proof}

\begin{proof}[Proof of Theorem~\ref{introthm:main}]
We apply Theorem~\ref{thm:main_perfect} to the base language $\Lang=\Lval$.
In this case the language on the residue field sort is simply $\Lk=\Lring$,
and then $\Th(k)$ of course admits a computable quantifier-elimination,
since $k$ is a finite field.
\end{proof}

\begin{theorem}\label{thm:FqQ}
For every prime power $q=p^{\ell}$.
There is a computable function $\epsilon:\Form(\Lval)\rightarrow\Epos(\Lring)$ such that
for every $\Lval$-formula $\varphi(\underline{x})$ with free variables from the sort $\bbK$
we have
\begin{align*}
	(\Hs{\FF_{q}}{\QQ},v_{t})\models\forall\underline{x}\;(\varphi(\underline{x})\leftrightarrow\epsilon\varphi(\underline{x})).
\end{align*}
\end{theorem}
\begin{proof}
As in the proof of Theorem~\ref{introthm:main}, we apply Theorem~\ref{thm:main} to the language $\Lang=\Lval$ and to the residue field $k$.
However, in this case the theory of the value group is the theory of divisible ordered abelian groups, which admits a computable quantifier elimination.
This yields a computable elimination down to the fragment generated by $\FEpos(\Lring)$ and the quantifier-free formulas in $\Lval$.
The latter are computably equivalent modulo $\H(\Th(\FF_{q}))$ to $\FEpos(\Lring)$-formulas by Theorem~\ref{thm:E1}.
We observe that $\wp_{q}-1$ is a polynomial over $\ZZ$ of degree $p$ with no root in any model of $\H(\Th(\FF_{q}))$.
Therefore, by Remark~\ref{rem:basic} applied to $f=\wp_{q}-1$, there is a computable elimination $\FEpos(\Lring)\rightarrow\Epos(\Lring)$ modulo $\H(\Th(\FF_{q}))$ as required.
\end{proof}

Theorem~\ref{introthm:FqQ} is an immediate corollary of Theorem~\ref{thm:FqQ}.
Theorem~\ref{thm:main} also yields the following corollary, which is simply a resplendent version of Lisinski's Theorem (Theorem~\ref{thm:Lisinski}), in the setting of finite residue fields.

\begin{corollary}\label{cor:Lisinski_style}
Let $\Lang=\Lval(\Lk,\LG)$ be a $(\bbk,\bbG)$-expansion of $\Lval$,
and let $k$ denote a presented $\Lk$-structure expanding a finite field.
	For every $(K,v)\models\TVF(\Th^{\Lk}(k),\emptyset)$,
	and $\pi\in\mathfrak{m}_{v}\setminus\{0\}$,
	we have the \Dtext\ equivalence
	$$\Th(K,v,\pi)\Deq\Th(vK,v(\pi))\Doplus\Th_{\Epos(\Lring)}(K,\pi).$$
\end{corollary}
\begin{proof}
	Observe that the valued subfield generated by $\pi$ is $\FF_{p}(\pi)$,
	and the value group with constants is $(vK,v(\pi))$.
\end{proof}

% \begin{remark}
% This corollary should admit a version for ``classical fragments'':
% if $F$ is a sequence of $\exists$ and $\forall$ then
% 	\SA{Carry on here}
% \end{remark}

\begin{question}
Can Theorem~\ref{thm:main} (or perhaps a weaker statement) be proved uniformly for the theory of finite residue fields?
Can this be proved for other theories of residue fields?
What about for the theory of prime fields?
\end{question}

\section{The additive fragment is NTP2}
\label{section:NTP2}

We now circle back to the question that has haunted us for 6 years:
\begin{question}
Is $\Hs{\FF_{q}}{\QQ}$ NTP2?
\end{question}

This question fits into the more general framework of classification/neo-stability, which explores links between model-theoretic dividing lines and algebraic properties of pure fields, valued fields, or other structures such as ordered fields or difference fields.
In chronological order, \cite{HHJ,AJ,B-chips} have established classifications of strongly dependent, NIP, and finally NIP$_{\!n}$ henselian valued fields, down to their residue field.

Our friend $\Hs{\FF_{q}}{\QQ}$ has IP$_{\!n}$, as it has one Artin-Schreier extension, see \cite{KSW,Hempel}. Thus, it is out of the scope of the existing classifications, as are most tame henselian valued fields. We believe that the correct model-theoretic dividing line to study those fields is NTP2, introduced by Shelah in \cite{Sh-NTP2}, which encompasses both Simple and NIP theories, and is defined as follows:

\begin{definition}
	Let $\Lang$ be a language, $T$ a complete $\Lang$-theory and $M\vDash T$ be $\aleph_1$-saturated. An $\Lang$-formula $\varphi(\underline x,\underline y)$ has the \define{tree property of the second kind} (in $T$), abbreviated TP2, if there is a natural number $k\geqslant 2$ and an array $(\underline{a_{ij}})_{i,j<\omega}$ with $a_{ij}\in M^{|\underline y|}$ such that $\{\varphi(\underline x,\underline {a_{ij}})\}_{j<\omega}$ is
$k$-inconsistent for every $i<\omega$ and $\{\varphi(\underline x,\underline{a_{i,f(i)}})\}_{i<\omega}$ is consistent for any $f \colon\omega\rightarrow\omega$. Otherwise, we say that $\varphi(\underline x,\underline y)$ is NTP2.
We also say that a theory $T$ is NTP2 if every formula is NTP2 in $T$, and that a structure $M$ is NTP2 if $\Th(M)$ is NTP2. 
\end{definition}
We refer to \cite{CherNTP2} for details on NTP2.

A generic goal is to obtain an optimal NTP2 transfer theorem, that is, a classification of NTP2 henselian valued field down to their residue field. The aforementioned work of Chernikov does so in equicharacteristic 0, see \cite[Theorem 7.6]{CherNTP2}, but for other characteristics, we are far from a full classification. The strongest known results can be found in \cite{B-phd}. Let us focus on what happens in equicharacteristic $p$:

\begin{proposition}[{\cite[Propositions 2.5.17 and 2.5.19]{B-phd}}]
Let $(K,v)$ be a henselian valued field of positive characteristic.
If $(K,v)$ is NTP2, then it is semitame in the sense of \cite{KR-dr}, in particular, its value group is $p$-divisible and its residue field is perfect.
If $(K,v)$ is separably algebraic maximal Kaplansky
and its residue field is NTP2, then $(K,v)$ is NTP2.
\end{proposition}

This led us to phrase the following:

\begin{conjecture}
Let $(K,v)$ be a henselian valued field of positive characteristic.
Then $(K,v)$ is NTP2 if and only if it is tame and its residue field is NTP2.
\end{conjecture}

The Hahn series field $\Hs{\FF_{q}}{\QQ}$ is the simplest case we could possibly study, as it is tame with divisible value group and finite residue field. Hence, a first step towards the conjecture is to determine whether $\Hs{\FF_{q}}{\QQ}$ is NTP2.
Despite all the efforts thrown at it, to this day, it still resits.

\begin{definition}
Let $p$ be a prime number.
The additive fragment $F_+$ is the $\Lring$-fragment generated by formulas of the form $\exists w\;P(\underline x,w)=y$, where $P\in\FF_p[\underline{x},w]$ is additive with respect to $w$.
It is a sub-fragment of the Galois fragment $\FEpos(\Lring)$.
\end{definition}

The work of Kaplan, Scanlon, and Wagner~\cite{KSW} established a clear link between field extensions generated by Artin-Schreier roots and dividing lines. Their work was then extended to NTP2 by Chernikov, Kaplan, and Simon. As noted in \cite[Section 2.5.8]{B-phd}, their methods can be extended to additive polynomials. More precisely:

\begin{lemma}\label{lem:NTP2}
Let $K$ be a field of characteristic $p>0$ such that, for each $n\in\NN_{>0}$, $K$ has finitely many Galois extensions of degree $p^n$.
Let $\psi(x;\underline y,z)$ be a formula of the form $\bigwedge_{i=1}^n(\exists w\;P_i(\underline y,w)=x-z)$, where $P_i(\underline y,w)\in\FF_p[\underline y,w]$ is additive with respect to $w$.
Then $\psi$ is NTP2.
\end{lemma}

\begin{proof}
We follow the strategy appearing in the proof of \cite[Corollary 4.8]{B-asext}. For each $\overline a\in K^{|y|}$, let $H_{\overline a}$ be $\bigcap_{i=1}^n P_i(\overline a,K)$. Since each $P_i$ is additive, $H_{a}$ is an additive subgroup of $K$. Algebraically speaking, the formula $\psi(x;\overline a,z)$ is saying ``$x\in H_{\overline a}+z$''. By assumption on $K$, $[K:H_{\overline a}]$ is bounded, and thus by \cite[Porism 4.6]{B-asext} $\psi$ is NTP2.
\end{proof}
 
Since disjunctions of NTP2 formulas are NTP2, we obtain the following:

\begin{corollary}
Let $K$ be a field of characteristic $p>0$ such that for each $n$, $K$ has finitely many Galois extension of degree $p^n$.
Then any formula in the additive fragment of $K$ is NTP2.
In particular, this holds for $\Hs{\FF_{q}}{\QQ}$.
\end{corollary}

Recall that by Theorem~\ref{thm:FqQ}, any $\Lval$-formula
is equivalent
in the theory of $\Hs{\FF_p}{\QQ}$
to a formula of the form $\exists y\;P(\underline x,y)=0$ for arbitrary polynomials $P\in\FF_p[\underline x,y]$.
Meanwhile, formulas in the additive fragment are of a similar form with the extra requirement that the polynomials are additive. In order to prove NTP2, we would need to reduce from the arbitrary case to the additive case, something that we cannot do as of now.

It is well-known that $\Hs{\FF_{q}}{\QQ}$ does {\bf not} admit quantifier elimination, either in the language of valued fields, or in the language $\Lac$ of ac-valued fields. 
In particular:
the image of the Artin--Schreier function $\wp_{q}:\Hs{\FF_{q}}{\QQ}\rightarrow\Hs{\FF_{q}}{\QQ}$ is not quantifier-free $\Lac$-definable.

\begin{question}
Does $\Hs{\FF_{q}}{\QQ}$ admit elimination down to the additive $\Lring$-fragment?
\end{question}

A positive answer to this question would imply that $\Hs{\FF_{q}}{\QQ}$ is NTP2, by Lemma~\ref{lem:NTP2}.

\section*{Acknowledgements}

The first author
was supported by
the ANR-DFG project ``AKE-PACT'' (ANR-24-CE92-0082)
and
by ``Investissement d'Avenir'' launched by the French Government and implemented by ANR (ANR-18-IdEx-0001) as part of its program ``Emergence''.

%%% BIBLIOGRAPHY %%%
\def\bibfont{\footnotesize}
\bibliographystyle{plain}

%%%%%%%%%%%%%%%%%%%%%%%%%%%%%%%%%%%%%%%%%%%%%%%%%%%%%%%%%%%%%%%%%%%%%%%%%%%%%%%%%%%%%%%%%%%%%
%%%%%%%%%%%%%%%%%%%%%%%%%%%%%%%%%%%%%%%%%%%%%%%%%%%%%%%%%%%%%%%%%%%%%%%%%%%%%%%%%%%%%%%%%%%%%
%%%%%%%%%%%%%%%%%%%%%%%%%%%%%%%%%%%%%%%%%%%%%%%%%%%%%%%%%%%%%%%%%%%%%%%%%%%%%%%%%%%%%%%%%%%%%
%%%%%%%%%%%%%%%%%%%%%%%%%%%%%%%%%%%%%%%%%%%%%%%%%%%%%%%%%%%%%%%%%%%%%%%%%%%%%%%%%%%%%%%%%%%%%
%%%%%%%%%%%%%%%%%%%%%%%%%%%%%%%%%%%%%%%%%%%%%%%%%%%%%%%%%%%%%%%%%%%%%%%%%%%%%%%%%%%%%%%%%%%%%
\end{document}